\documentclass[letterpaper,10pt,conference]{ieeeconf}

\IEEEoverridecommandlockouts
\newtheorem{theorem}{Theorem}[section]

\newtheorem{remark}[theorem]{Remark}

\newtheorem{problem}[theorem]{Problem}
\newtheorem{definition}[theorem]{Definition}
\newtheorem{assumption}[theorem]{Assumption}

\usepackage{csquotes}
\usepackage{lettrine}
\usepackage{dirtytalk}
\usepackage{algorithm,algorithmic}
\usepackage{amssymb}
\usepackage{graphicx}
\usepackage[cmex10]{amsmath}
\usepackage{array}
\usepackage{mdwtab}
\usepackage{fixltx2e}
\usepackage{graphics}
\usepackage{epsfig}
\usepackage{times}
\usepackage{amsmath}
\usepackage{amssymb}
\usepackage{url}
\usepackage{lscape}
\usepackage{color}
\usepackage{epsfig}
\usepackage{cite}
\usepackage{balance}
\usepackage{dblfloatfix}
\usepackage{hyperref}
\usepackage{caption}
\usepackage{subcaption}

\allowdisplaybreaks

\title{\LARGE \bf Robust Steady-State-Aware Model Predictive Control for Systems with Limited Computational Resources and External Disturbances}

\author{Hassan Jafari Ozoumchelooei and Mehdi Hosseinzadeh,~\IEEEmembership{Senior Member,~IEEE}
\thanks{This work was supported by the WSU Voiland College of Engineering and Architecture through a start-up package provided to Mehdi Hosseinzadeh.}
\thanks{The authors are with the School of Mechanical and Materials Engineering, Washington State University, Pullman, WA 99164, USA (email: h.jafariozoumcheloo@wsu.edu, mehdi.hosseinzadeh@wsu.edu).}
}

\begin{document}

\maketitle
\thispagestyle{empty}
\pagestyle{empty}

\begin{abstract}
Model Predictive Control (MPC) is a powerful control strategy; however, its reliance on online optimization poses significant challenges for implementation on systems with limited computational resources. One possible approach to address this issue is to shorten the prediction horizon and adjust the conventional MPC formulation to enlarge the region of attraction. However, these methods typically introduce additional computational load. Recently, steady-state-aware MPC has been introduced to ensure output tracking and convergence to a given desired steady-state configuration while maintaining constraint satisfaction at all times without adding extra computational load. Despite its promising performance, steady-state-aware MPC does not account for external disturbances, which can significantly limit its applicability to real-world systems. This paper aims to advance the method further by enhancing its robustness against external disturbances. To achieve this, we adopt the tube-based design framework, which decouples nominal trajectory optimization from robust control synthesis, thereby requiring no additional online computational resources. Theoretical guarantees of the proposed methodology are shown analytically, and its effectiveness is assessed through simulations and experimental studies on a Parrot Bebop 2 drone.


\end{abstract}

\section{Introduction}

\lettrine{M}{odel} Predictive Control (MPC) is a widely used strategy for controlling systems with constraints on both states and control inputs \cite{camacho2013model,rawlings2017model}. At each time step, MPC solves an optimization problem over a finite prediction horizon to determine the optimal control input. Despite its effectiveness, MPC relies on online optimization, which can be computationally challenging for systems with limited computational resources \cite{hosseinzadeh2022rotec}.

Approaches to address the issue of limited computational power have been investigated over the last several decades.  Methods proposed include pre-computing the optimal solutions for different operating conditions offline and storing them for online use \cite{alessio2009survey}; using neural networks to approximate the optimal solution \cite{Wang2022,Peng2019}; using triggering mechanisms to determine when the optimization problem needs to be solved \cite{Henriksson2012,wang2020event}; and relying on inexact computations (e.g., with a fixed number of optimization iterations) to obtain a sub-optimal approximation of the solution \cite{Ghaemi2009,Cimini2017}. 

An alternative approach is to shorten the prediction horizon, thereby reducing the dimension of the optimization problem; this approach has been explored in \cite{sawma2018effect,pannek2011reducing}. Although these methods guarantee constraint satisfaction at all times, they do not simultaneously ensure output tracking and convergence to a given steady-state configuration. This issue has been addressed in \cite{amiri2024steady}. However, \cite{amiri2024steady} does not account for external disturbances, which limits its applicability and practicality in real-world scenarios. This paper aims to address this issue by further advancing the method presented in \cite{amiri2024steady}.

To achieve this, we employ tube-based MPC methodology \cite{mayne2009robust,limon2010robust,trodden2010distributed,gonzalez2011online}, which uses set-theoretic techniques to ensure that all possible trajectories of a system impacted by external disturbances stay within a series of acceptable regions known as reachability tubes. Our motivation for building upon tube-based techniques is their ability to decouple nominal trajectory optimization from robust control synthesis, thereby reducing computational demands for handling disturbances.

The key contributions of this paper are: i) the development of a novel robust MPC that can handle external disturbances, ensure constraint satisfaction, and address control objectives for any given prediction horizon length; ii) the provision of theoretical guarantees for its recursive feasibility and closed-loop stability; and iii) the analysis and validation of the proposed scheme through extensive simulation and experimental studies.

The rest of this paper is organized as follows. Section \ref{sec:PS} presents the preliminary materials and states the problem. The proposed control scheme is developed in Section \ref{sec:RSSMPC}, and its theoretical properties are investigated in Section \ref{secLProperties}. Sections \ref{sec:simulation} and \ref{sec:Experimental} present simulation and experimental results, respectively. Finally, Section \ref{sec:conclusion} concludes the paper.

\medskip\noindent
\textbf{Notations:} We use $\mathbb{R}$ to denote the set of real numbers, $\mathbb{R}_{>0}$ to denote the set of positive real numbers, and $\mathbb{R}_{\geq0}$ to denote the set of non-negative real numbers. Similarly, we use $\mathbb{Z}$, $\mathbb{Z}_{>0}$, and $\mathbb{Z}_{\geq0}$ to denote the sets of integers, positive integers, and non-negative integers, respectively. We use $x \in \mathbb{R}^n$ to indicate that $x$ belongs to the $n$-dimensional real space. Given $x \in \mathbb{R}^n$, we denote its transpose by $x^\top$. Also, $\left\Vert x \right\Vert_Q^2 = x^\top Q x$, where $Q \in \mathbb{R}^{n \times n}$. We use $\ast$ in the superscript to indicate an optimal decision. We use $\text{diag}\{a_1, a_2, \dots, a_n\}$ to denote an $n \times n$ matrix with diagonal entries $a_1, a_2, \dots, a_n \in \mathbb{R}$. For given sets $X, Y \subset \mathbb{R}^n$, $X \ominus Y := \{x \mid x - y \in X, \forall y \in Y\}$ is the Pontryagin set difference, and $X \oplus Y := \{x + y \mid x \in X, y \in Y\}$ is the Minkowski set sum. We use $I_n$ to denote the $n \times n$ identity matrix, and $\mathbf{0}$ to denote a zero matrix of appropriate dimensions.

\section{Preliminaries and Problem Statement}\label{sec:PS}
\noindent\textbf{Problem Setting:} Consider that the system to be controlled can be represented by the following disturbed discrete-time linear time-invariant model:
\begin{subequations}\label{eq1}
\begin{align}
x(t + 1) &= Ax(t) + Bu(t) + w(t), \\
y(t) &= Cx(t) + Du(t), 
\end{align}
\end{subequations}
where $x(t) \in \mathbb{R}^n$ is the state vector, $u(t) \in \mathbb{R}^p$ is the control input, $y(t) \in \mathbb{R}^m$ is the output vector, and $w(t)\in \mathcal{W}\subset\mathbb{R}^n$ is an unknown bounded disturbance vector. 

\begin{assumption}\label{assumption1}
The pair \((A, B)\) is stabilizable.
\end{assumption}

At any time instant $t$, the state and input must satisfy the following constraints:
\begin{equation}\label{eq2}
x(t) \in \mathcal{X},~~u(t) \in\mathcal{U},
\end{equation}
where $\mathcal{X} \subseteq \mathbb{R}^n$ and $\mathcal{U} \subseteq \mathbb{R}^p$ are convex and compact sets containing the origin. The disturbance-free (a.k.a. nominal) dynamics of system \eqref{eq1} can be expressed as follows:
\begin{subequations}\label{eq3}
\begin{align}
\bar{x}(t + 1) &= A\bar{x}(t) + B\bar{u}(t), \\
\bar{y}(t) &= C\bar{x}(t) + D\bar{u}(t), 
\end{align}
\end{subequations}where \(\bar{x}\), \(\bar{u}\) and \(\bar{y}\) are nominal state, input and output respectively. For any given reference $r \in \mathbb{R}^m$, suppose that there exists at least one steady-state pair $\big(\bar{x}_s, \bar{u}_s\big)$ such that
\begin{equation}
\bar{x}_s=A\bar{x}_s+B\bar{u}_s,~~r= C\bar{x}_s + D\bar{u}_s,
\label{eq4}
\end{equation}
where $\bar{x}_s \in\mathcal{X}$ and $\bar{u}_s \in\mathcal{U}$. A reference that meets this criterion is called a steady-state admissible reference; we denote the set of all such references by $\mathcal{R} \subseteq \mathbb{R}^m$. Thus, for any $r\in\mathcal{R}$, the set of all admissible steady-state configurations is:
\begin{align}
\mathcal{Z}_r=\big\{ (\bar{x}_s,\bar{u}_s)\in\mathcal{X}\times\mathcal{U}| &\bar{x}_s = A \bar{x}_s + B \bar{u}_s,\nonumber\\
&r=C \bar{x}_s + D \bar{u}_s\big\}.
\end{align}

According to Assumption \ref{assumption1}, elements of $\mathcal{Z}_r$ can be parameterized as \cite{limon2008mpc,amiri2024steady,Ferramosca2008,Ferramosca2009}:
\begin{align}
\bar{x}_s=M_1\theta,~\bar{u}_s=M_2\theta,~ r=L\theta,
\label{steady_state}
\end{align}
Where $M_1 \in \mathbb{R}^{n \times n_\theta}$, $M_2 \in \mathbb{R}^{p \times n_\theta}$, $L \in \mathbb{R}^{m \times n_\theta}$, and $\theta \in \mathbb{R}^{n_\theta}$ is the parameterizing vector, with $n_\theta$ denoting the dimensionality of $\theta$. See \cite{amiri2024steady} for guidelines on determining the matrices $M_1$, $M_2$, and $L$.



\medskip
\noindent\textbf{Tube of Trajectories:} Suppose that the control input $\bar{u}(t)$ is calculated at time instant $t$ for the nominal system given in \eqref{eq3}. Considering that the actual system may be disturbed by the external disturbance $w(t)$, the trajectory of the disturbed system is likely to deviate from the nominal prediction. To mitigate the impact of these disturbances and keep the trajectory as close to the nominal path as possible, we can determine the control input applied to the system, denoted by $u(t)$, as follows \cite{mayne2005robust,limon2010robust}:
\begin{equation}\label{eq_u}
u(t) = \bar{u}(t) + K \big(x(t) - \bar{x}(t)\big),
\end{equation}
where $K\in\mathbb{R}^{p\times n}$ is such that $A+BK$ is Schur. In this case, according to \eqref{eq1} and \eqref{eq3}, the dynamics of the error signal $e(t)\triangleq x(t) - \bar{x}(t)$ are:
\begin{equation}\label{err}
e(t+1) = (A + B K) e(t) + w(t). 
\end{equation}

Since $A+BK$ is Schur and the disturbance signal $w(t)$ is bounded, \eqref{err} implies that $e(t)$ is bounded, and thus the actual trajectory remains within a neighborhood around the nominal trajectory; such a neighborhood can be seen \cite{mayne2001robustifying} as a tube of trajectories around the nominal trajectory.

At this stage, we define the concept of a Robustly Positively Invariant (RPI) set \cite{kolmanovsky1998theory}, which will be used later to tighten the constraints and account for the external disturbance signal $w(t)$. 

\begin{definition}
The set $\mathcal{F} \subseteq \mathbb{R}^n$ is called a RPI set for system \eqref{err} if, for any $e(t) \in \mathcal{F}$ and $w(t) \in \mathcal{W}$, we have $e(k+1) \in \mathcal{F}$. The minimal RPI set for system \eqref{err}, denoted by $\mathcal{F}_\infty$, is given by \cite{mayne2005robust} as $\mathcal{F}_\infty = \bigoplus_{i=0}^{\infty} (A+BK)^i \mathcal{W}$.
\end{definition}

\begin{remark}
In general, it is impossible to obtain an explicit characterization of $\mathcal{F}_\infty$. However, one can follow the procedure detailed in \cite{Rakovic2004} to obtain an invariant approximation of the minimal RPI $\mathcal{F}_\infty$. Note that the computation of $\mathcal{F}_\infty$ is performed offline and thus does not add an online computational load to the MPC problem.
\end{remark}

\medskip\noindent\textbf{Goal:} This paper aims to address the following problem. 

\begin{problem}
Given $r \in \mathcal{R}$, the desired steady-state configuration $\big(\bar{x}_{\text{des}}, \bar{u}_{\text{des}}\big) \in \mathbb{R}^n \times \mathbb{R}^p$, and the initial condition $x(0) \in \mathcal{X}$, obtain a robust optimal control input that guides the output of system \eqref{eq1} to $r$ and achieves the steady-state configuration $\big(\bar{x}_{\text{des}}, \bar{u}_{\text{des}}\big)$, while satisfying the constraints \eqref{eq2}, despite the presence of the disturbance signal $w(t)$. 
\label{prob1}
\end{problem}

\section{Robust Steady-State-Aware Model Predictive Control}\label{sec:RSSMPC}
In this section, we propose a Robust Steady-State-Aware Model Predictive Control (RSSA-MPC) scheme to address Problem~\ref{prob1}, which is developed by building upon the recently introduced Steady-State-Aware Model Predictive Control (SSA-MPC) \cite{amiri2024steady}. The proposed RSSA-MPC scheme ensures tracking of piecewise constant references, convergence to the desired steady-state configuration, and satisfaction of state and input constraints without requiring prior knowledge of the reference and despite the presence of external disturbances.


Let $N \in \mathbb{Z}_{>0}$ be the length of the prediction horizon. At any time instant  $t$, the proposed RSSA-MPC determines the optimal parameterizing vector $\theta^\ast(t) \in \mathbb{R}^{n_\theta}$, the optimal initial condition for the nominal system $\bar{x}^\ast_0(t)$, and the optimal control sequence $\mathbf{\bar{u}}^\ast(t):=\Big[\big(\bar{u}^\ast(0|t)\big)^\top,\cdots,\big(\bar{u}^\ast(N - 1|t)\big)^\top\Big]^\top \in \mathbb{R}^{N\times p}$ by solving the following optimization problem:
\begin{subequations}\label{eq:OptimizitonProblem}
\begin{align}\label{opt}
&\theta^\ast(t),\bar{x}^\ast_0(t), \mathbf{\bar{u}}^\ast(t) =\arg\min_{\theta,\bar{x}_0,\mathbf{\bar{u}}} \Bigg( \sum_{k=0}^{N-1} \| \hat{\bar{x}}(k|t) - M_1 \theta \|^2_{Q_x} \nonumber\\
& \quad + \sum_{k=0}^{N-1} \| \bar{u}(k|t) - M_2 \theta \|^2_{Q_u} + \| \hat{\bar{x}}(N|t) - M_1 \theta \|^2_{Q_N} \nonumber\\
& \quad + \| r - L \theta \|^2_{Q_r} + \left\Vert \bar{x}_s-\bar{x}_{\text{des}}\right\Vert_{Q_{sx}}^2+\left\Vert \bar{u}_s-\bar{u}_{\text{des}}\right\Vert_{Q_{su}}^2 \Bigg),
\end{align}
subject to the following constraints:
\begin{align}
&\hat{\bar{x}}(k+1|t)=A\hat{\bar{x}}(k|t)+B\bar{u}(k|t),~\hat{\bar{x}}(0|t)=\bar{x}_0\\
&\bar{x}_0 \in x(t)\oplus\mathcal{F}_{\infty},\\
&\hat{\bar{x}}(k|t)  \in \mathcal{X} \ominus \mathcal{F}_{\infty}, \quad k \in \{0, \ldots, N - 1\},\\
&\bar{u}(k|t) \in \mathcal{U} \ominus K\mathcal{F}_{\infty}, \quad k\in \{0, \ldots, N - 1\},\\
&\big(\hat{\bar{x}}(N|t), \theta\big) \in \Omega. \label{t_cons}
\end{align}
\end{subequations}

In \eqref{eq:OptimizitonProblem}, $\hat{\bar{x}}(k|t)$ represents the predicted nominal state at instant $k$, $Q_x=Q_x^{\top} \succeq 0$ ($Q_x\in\mathbb{R}^{n\times n}$), $Q_u=Q_u^{\top} \succ 0$ ($Q_u\in\mathbb{R}^{p\times p}$), $Q_N\succeq 0$ ($Q_N\in\mathbb{R}^{n\times n}$), $Q_r=Q_r^{\top} \succ 0$ ($Q_r\in\mathbb{R}^{m\times m}$), $Q_{sx}=Q_{sx}^{\top} \succeq 0$ ($Q_{sx} \in \mathbb{R}^{n \times n}$), $Q_{su}=Q_{su}^{\top} \succeq 0$ ($Q_{su} \in \mathbb{R}^{p \times p}$), and $\bar{x}_{\text{des}}\in\mathbb{R}^n$ and $\bar{u}_{\text{des}}\in\mathbb{R}^p$ are the desired steady state and control input determined by the designer. Also, given a terminal control law $\bar{u}(t) = \kappa\big(\hat{\bar{x}}(k|t), \theta\big)$ with $\kappa : \mathbb{R}^n \times \mathbb{R}^{n_\theta} \rightarrow \mathbb{R}^p$, the robust terminal constraint set $\Omega \subseteq \mathbb{R}^n \times \mathbb{R}^{n_\theta}$ is designed to ensure that if $\big(\hat{\bar{x}}(N|t), \theta\big) \in \Omega$, then $\Big(\hat{\bar{x}}(k|t), \kappa\big(\hat{\bar{x}}(k|t), \theta\big)\Big) \in (\mathcal{X} \ominus \mathcal{F}_{\infty} \times \mathcal{U} \ominus K \mathcal{F}_{\infty})$ for all $k \geq N$. By $K\mathcal{F}_{\infty}$, we mean the matrix multiplication of $K$ by all elements of the set $\mathcal{F}_{\infty}$.

\begin{remark}\label{remarkQN}
In line with conventional MPC \cite{nicotra2018embedding, hosseinzadeh2023robust,limon2010robust}, it is effective to select \(Q_{N}\) in \eqref{opt} as the solution to the algebraic Riccati equation $Q_N = A^\top Q_N A - (A^\top Q_N B)(Q_u + B^\top Q_N B)^{-1}(B^\top Q_N A) + Q_x$, and to use the terminal control law $\kappa(\bar{x}(t), \theta) = M_2 \theta + K_{\infty}(\bar{x}(t) - M_1 \theta)$, where $K_{\infty}=-(Q_u + B^\top Q_N B)^{-1} B^\top Q_N A$. It is easy to show that $A+BK_\infty$ is Schur. 
\end{remark}


\subsection{Robust Terminal Constraint Set $\Omega$}\label{AA}

When nominal system~\eqref{eq3} is controlled by the terminal control law $\kappa(\bar{x}(t), \theta) = M_2 \theta + K_{\infty}(\bar{x}(t) - M_1 \theta)$, it evolves as follows:
\begin{equation}\label{x_lqr_1}
\bar{x}(t + 1) = (A + BK_{\infty})\bar{x}(t) + (BM_2 - BK_\infty M_1) \theta.
\end{equation}

Thus, the robust terminal constraint set $\Omega$ can be implemented as $\Omega= \mathcal{O}^{\epsilon}\cap\mathcal{O}_{\infty}$, where $\mathcal{O}^{\epsilon}$ is defined as:
\begin{align}
\mathcal{O}^{\epsilon}=\big\{\theta|&M_1 \theta \in (1 - \epsilon) (\mathcal{X} \ominus \mathcal{F}_{\infty}),\nonumber\\
&M_2 \theta \in (1 - \epsilon) (\mathcal{U} \ominus K \mathcal{F}_{\infty})\big\},
\end{align}
for some $\epsilon\in(0,1)$, and $\mathcal{O}_{\infty}$ is the tightened maximal output admissible set defined as:
\begin{align}
&\mathcal{O}_{\infty}  = \{(\bar{x}, \theta) \mid 
\hat{\bar{x}}( \gamma \mid \bar{x}, \theta) \in \mathcal{X} \ominus \mathcal{F}_{\infty},~M_2\theta \nonumber\\
&+K_\infty\left(\hat{\bar{x}}(\gamma|\bar{x},\theta)-M_1\theta\right) \in \mathcal{U} \ominus K \mathcal{F}_{\infty}, \, \gamma = 0, 1,\cdots \},
\end{align}
with $\hat{\bar{x}}(\gamma|\bar{x},\theta)$ being the predicted state at the prediction instant $\gamma$ starting from $\bar{x}$ with $\theta$ being kept constant, which, according to \eqref{x_lqr_1}, can be computed as follows: 
\begin{align}
 &\hat{\bar{x}}(\gamma \mid \bar{x}, \theta) =  (A + BK_{\infty})^{\gamma} \bar{x} \nonumber\\
& ~~~~~~+\sum_{i=1}^{\gamma} (A + BK_{\infty})^{i-1} (BM_2 - BK_{\infty} M_1) \theta,\label{x_hat_futur}
\end{align}



Note that since $\mathcal{O}_\infty$ is bounded (see \cite{amiri2024steady} and \cite{gilbert1991linear}), the robust terminal constraint set defined above is finitely determined. That is, there exists a finite index $\gamma^\ast$ such that $\mathcal{O}_{\gamma}=\mathcal{O}_{\infty},~\gamma\geq\gamma^\ast$, and thus the robust terminal constraint set can be defined as $\Omega=\mathcal{O}^{\epsilon}\cap\mathcal{O}_{\gamma^\ast}$. The index $\gamma^\ast$ can be obtained by solving a sequence of programming problems as discussed in \cite{gilbert1991linear}.


\section{Theoretical Properties of the Proposed RSSA-MPC}\label{secLProperties}
In this section, we will investigate the theoretical properties of the proposed RSSA-MPC. First, we will show that the proposed scheme is recursively feasible; starting from a feasible initial condition, the optimization problem \eqref{eq:OptimizitonProblem} remains feasible at all times. Next, we will show that the control inputs obtained from the optimization problem \eqref{eq:OptimizitonProblem} steer the output of system \eqref{eq1} toward the desired reference without violating the constraints given in \eqref{eq2}.

\begin{theorem} \label{Theorem1}
Consider the system described by \eqref{eq1} and subject to the constraints specified in \eqref{eq2}. Assume that the optimization problem \eqref{eq:OptimizitonProblem} is feasible at the time instant $t$. Then, it remains feasible for all $t\in\mathbb{Z}_{\geq0}$.
\end{theorem}

\begin{proof}
Suppose that at time instant $t$, the RSSA-MPC problem given in \eqref{eq:OptimizitonProblem} is feasible. Thus, the optimal parameterizing vector $\theta^\ast(t)$, the optimal initial condition $\bar{x}_0^\ast(t)$, and the optimal control sequence $\mathbf{\bar{u}}^\ast(t)$ are available. Define the following sequences:
\begin{subequations}\label{eq:Feasiblet+1}
\begin{align}
\theta^\dag=&\theta^\ast(t),\\
\bar{x}_0^\dag=&A\bar{x}_0^\ast+B\bar{u}^\ast(0|t),\\
\mathbf{\bar{u}}^\dag=&\Big[\big(\bar{u}^\ast(1|t)\big)^\top,\cdots,\big(\bar{u}^\ast(N - 1|t)\big)^\top \nonumber\\
&~~\big(\kappa\big(\hat{\bar{x}}(N|t), \theta^\ast(t)\big)\big)^\top\Big]^\top.
\end{align} 
\end{subequations}

The $\big(\theta^\dag,\bar{x}_0^\dag,\mathbf{\bar{u}}^\dag\big)$ given in \eqref{eq:Feasiblet+1} is a feasible solution for the optimization problem \eqref{eq:OptimizitonProblem} at time instant $t+1$, as: i) we know that $x(t+1)\in\left(A\bar{x}_0^\ast+B\bar{u}^\ast(0|t)\right)\oplus\mathcal{F}_{\infty}$, which according to the fact that $\mathbf{0}\in\mathcal{F}_{\infty}$ implies $A\bar{x}_0^\ast+B\bar{u}^\ast(0|t)\in x(t+1)\oplus\mathcal{F}_{\infty}=\big((A\bar{x}_0^\ast+B\bar{u}^\ast(0|t))\oplus\mathcal{F}_{\infty}\big)\oplus\mathcal{F}_{\infty}$; and ii) the terminal constraint set $\Omega$ is positively invariant, which implies that $\kappa(\hat{\bar{x}}(N|t), \theta^\ast(t)) \in \mathcal{U}\ominus K\mathcal{F}_{\infty}$ and $\hat{\bar{x}}(N + 1|t) = A \hat{\bar{x}}(N|t) + B \kappa(\hat{\bar{x}}(N|t), \theta^\ast(t)) \in \mathcal{X} \ominus \mathcal{F}_{\infty}$. Thus, feasibility at time instant $t$ implies feasibility at time instant $t+1$, indicating that the proposed RSSA-MPC is recursively feasible; this completes the proof.
\end{proof}

Now, we study the closed-loop stability of the proposed RSSA-MPC and show that the steady-state configuration converges to the desired configuration if it is admissible, or else to the \textit{best} admissible configuration, while ensuring output tracking and constraint satisfaction at all times.

\begin{theorem}
Consider system \eqref{eq1} and suppose that the control input is obtained by solving the optimization problem \eqref{eq:OptimizitonProblem}. Then, if \( (\bar{x}_{\text{des}}, \bar{u}_{\text{des}}) \in (\mathcal{X} \ominus \mathcal{F}_{\infty}\times\ \mathcal{U} \ominus K\mathcal{F}_{\infty})\bigcap\mathcal{Z}_r \), the system output \( \bar{y}(t) \) will converge to \( r \), and the state \( (\bar{x}_s, \bar{u}_s) \) will approach \( (\bar{x}_{\text{des}}, \bar{u}_{\text{des}}) \) as \( k \to \infty \). Otherwise, the system output \(\bar{y}(t)\) will converge to \( L \Tilde{\theta} \), and the state \( (\bar{x}_s, \bar{u}_s) \) will approach \( (M_1 \Tilde{\theta}, M_2 \Tilde{\theta}) \) as \( k \to \infty \), where \( \Tilde{\theta} \) satisfies:
\begin{equation}
\begin{aligned}
\| \theta^\diamond - \Tilde{\theta} \| \leq \frac{f(\Tilde{\theta})}{\alpha}
\end{aligned}
\label{eq_bound}
\end{equation}
for some $\alpha\in\mathbb{R}_{\geq0}$, with $f(\theta)$ being a positive definite function such that $f(\theta)=0\Leftrightarrow\left(M_1\theta,M_2\theta\right)=\left(\bar{x}_{\text{des}},\bar{u}_{\text{des}}\right)$, \( \theta^\diamond \) being the solution to the following optimization problem:
\begin{align}\label{eq:ThetaDimond}
 \theta^\diamond=\left\{
\begin{array}{ll} 
    & \arg\min_{\theta} \left\Vert M_1\theta-\bar{x}_{\text{des}}\right\Vert_{Q_{sx}}^2\\
    &~~~~~~~~~+\left\Vert M_2\theta-\bar{u}_{\text{des}}\right\Vert_{Q_{su}}^2\\
\text{s.t.} &  (M_1\theta, M_2\theta) \in (\mathcal{X} \ominus \mathcal{F}_{\infty}\times\mathcal{U} \ominus K\mathcal{F}_{\infty})\bigcap\mathcal{Z}_r
\end{array}
 \right.
\end{align}
\end{theorem}

\begin{proof}
Let $J\big(\theta, \bar{x}_0, \bar{u} \mid x(t)\big)$ and $J\big(\theta, \bar{x}_0, \bar{u} \mid x(t + 1)\big)$ denote the cost functions of the RSSA-MPC at time instants $t$ and $t+1$, respectively. Hence, $J\big(\theta^\ast(t), \bar{x}_0^\ast(t), \mathbf{\bar{u}}^\ast(t)|x(t)\big)$ and $J\big(\theta^\ast(t+1), \bar{x}_0^\ast(t+1), \mathbf{\bar{u}}^\ast(t+1)|x(t+1)\big)$ represent the optimal solutions at times $t$ and $t+1$, respectively. 

Optimality of $\big(\theta^\ast(t + 1), \bar{x}_0^\ast(t+1), \mathbf{\bar{u}}^\ast(t + 1)\big)$ yields:
\begin{equation}
\begin{aligned}
J\big(\theta^\ast(t + 1), \bar{x}_0^\ast(t+1), \mathbf{\bar{u}}^\ast(t + 1) \mid x(t + 1)\big) &   \\
\leq J\big(\theta^\ast(t), \bar{x}_0^\ast(t), \mathbf{\bar{u}}^\ast(t + 1) \mid x(t + 1)\big).
\end{aligned}
\end{equation}

Subtracting $J\big(\theta^\ast(t), \bar{x}_0^\ast(t), \mathbf{\bar{u}}^\ast(t) \mid x(t)\big)$ from both sides of this last inequality gives:
\begin{equation}\label{eq:Proof1}
\begin{aligned}
J\big(\theta^\ast(t + 1), \bar{x}_0^\ast(t+1), \mathbf{\bar{u}}^\ast(t + 1) \mid x(t + 1)\big)&\\
- J\big(\theta^\ast(t), \bar{x}_0^\ast(t), \mathbf{\bar{u}}^\ast(t) \mid x(t)\big) & \\
\leq J\big(\theta^\ast(t), \bar{x}_0^\ast(t), \mathbf{\bar{u}}^\ast(t + 1) \mid x(t + 1)\big)&\\
- J\big(\theta^\ast(t), \bar{x}_0^\ast(t), \mathbf{\bar{u}}^\ast(t) \mid x(t)\big).
\end{aligned}
\end{equation}

It is obvious that when $\bar{x}_0^\ast(t)=\bar{x}_0^\ast(t+1)$, we have $\bar{u}^\ast(k + 1 \mid t) = \bar{u}^\ast(k \mid t + 1)$ and $\hat{\bar{x}}(k + 1 \mid t) = \hat{\bar{x}}(k \mid t + 1)$, for $k = 0, \cdots, N - 2$. Thus, according to \eqref{opt}, it follows from \eqref{eq:Proof1} that:
\begin{align}
&J\big(\theta^\ast(t\!+\!1), \bar{x}_0^\ast(t\!+\!1), \mathbf{\bar{u}}^\ast(t\!+\!1) \mid x(t\!+\!1)\big) \nonumber\\
&\quad - J\big(\theta^\ast(t), \bar{x}_0^\ast(t), \mathbf{\bar{u}}^\ast(t) \mid x(t)\big) \leq \nonumber \\
&\left\| \hat{\bar{x}}(N \!\mid\! t\!+\!1) \!-\! M_1 \theta^\ast(t) \right\|^2_{Q_N}- \left\| \hat{\bar{x}}(N \!\mid\! t) \!-\! M_1 \theta^\ast(t) \right\|^2_{Q_N} \nonumber\\
&+ \left\| \hat{\bar{x}}(N\!-\!1 \!\mid\! t\!+\!1) \!-\! M_1 \theta^\ast(t) \right\|^2_{Q_x} \nonumber\\
&+ \left\| \bar{u}(N\!-\!1 \!\mid\! t\!+\!1) \!-\! M_2 \theta^\ast(t) \right\|^2_{Q_u} \nonumber\\
& - \left\| \hat{\bar{x}}^\ast(0 \!\mid\! t) \!-\! M_1 \theta^\ast(t) \right\|^2_{Q_x}- \left\| \bar{u}^\ast(0 \!\mid\! t) \!-\! M_2 \theta^\ast(t) \right\|^2_{Q_u}.
\label{closed-loop-stability}
\end{align}

As shown in~\cite{mayne2000constrained,amiri2024steady}, when \( Q_N \) is derived from the algebraic Riccati equation mentioned in Remark \ref{remarkQN}, The right-hand side of inequality \eqref{closed-loop-stability} is less than or equal to zero. Thus, we have:
\begin{equation}
\begin{split}
&J\big(\theta^\ast(t\!+\!1), \bar{x}_0^\ast(t\!+\!1), \mathbf{\bar{u}}^\ast(t\!+\!1) \mid x(t\!+\!1)\big) \\
&\quad - J\big(\theta^\ast(t), \bar{x}_0^\ast(t), \mathbf{\bar{u}}^\ast(t) \mid x(t)\big) \leq \\
&- \left\| \hat{\bar{x}}^\ast(0 \!\mid\! t) \!-\! M_1 \theta^\ast(t) \right\|^2_{Q_x} \\
&\quad - \left\| \bar{u}^\ast(0 \!\mid\! t) \!-\! M_2 \theta^\ast(t) \right\|^2_{Q_u}\leq0.
\end{split}
\label{closed-loop-stability2}
\end{equation}

Next, we will show that the only entire trajectory that satisfies $J\big(\theta^\ast(t+1), \bar{x}_0^\ast(t+1), \mathbf{\bar{u}}^\ast(t+1) \mid x(t+1)\big)-J\big(\theta^\ast(t), \bar{x}_0^\ast(t), \mathbf{\bar{u}}^\ast(t) \mid x(t)\big)\equiv0$ is the desired steady-state configuration $(\bar{x}_{\text{des}},\bar{u}_{\text{des}})$ if  \( (\bar{x}_{\text{des}},\bar{u}_{\text{des}}) \in (\mathcal{X} \ominus \mathcal{F}_{\infty}\times \mathcal{U} \ominus K\mathcal{F}_{\infty})\bigcap\mathcal{Z}_r\), and 
is  \( (M_1 \Tilde{\theta}, M_2 \Tilde{\theta}) \) otherwise, where $\Tilde{\theta}$ is as in \eqref{eq_bound}.

Suppose that there exists a time instant $t^\circ$ such that $J\big(\theta^\ast(t+1), \bar{x}_0^\ast(t+1), \mathbf{\bar{u}}^\ast(t+1) \mid x(t+1)\big)-J\big(\theta^\ast(t), \bar{x}_0^\ast(t), \mathbf{\bar{u}}^\ast(t) \mid x(t)\big)=0$ for $t\geq t^\circ$. Following arguments similar to \cite{amiri2024steady}, it can be shown that 
\begin{align}\label{eq:Proof2}
&J\big(\theta^\ast(t), \bar{x}_0^\ast(t), \mathbf{\bar{u}}^\ast(t) \mid x(t)\big)=\| r - L \theta^\ast(t) \|^2_{Q_r} \nonumber\\
&+ \left\Vert M_1\theta^\ast(t)-\bar{x}_{\text{des}}\right\Vert_{Q_{sx}}^2+\left\Vert M_2\theta^\ast(t)-\bar{u}_{\text{des}}\right\Vert_{Q_{su}}^2,
\end{align}
for all $t\geq t^\circ$. Also, optimality at time instant $t^\circ$ implies that:
\begin{align}\label{eq:Proof3}
J\big(\theta^\ast(t), \bar{x}_0^\ast(t), \mathbf{\bar{u}}^\ast(t) \mid x(t)\big)\leq J\big(\theta^\diamond(t), \bar{x}_0^\ast(t), \mathbf{\bar{u}}^\ast(t) \mid x(t)\big),
\end{align}
where $\theta^\diamond$ is as in \eqref{eq:ThetaDimond}. Since $(M_1\theta^\diamond, M_2\theta^\diamond) \in\mathcal{Z}_r$, it follows from \eqref{eq:Proof2} and \eqref{eq:Proof3} that for $t\geq t^\circ$ we have:
\begin{align}\label{eq:Proof4}
&\| r - L \theta^\ast(t) \|^2_{Q_r}+ \left\Vert M_1\theta^\ast(t)-\bar{x}_{\text{des}}\right\Vert_{Q_{sx}}^2\nonumber\\
&+\left\Vert M_2\theta^\ast(t)-\bar{u}_{\text{des}}\right\Vert_{Q_{su}}^2\leq\left\Vert M_1\theta^\diamond-\bar{x}_{\text{des}}\right\Vert_{Q_{sx}}^2\nonumber\\
&+\left\Vert M_2\theta^\diamond-\bar{u}_{\text{des}}\right\Vert_{Q_{su}}^2.
\end{align}

First, suppose that \( (\bar{x}_{\text{des}}, \bar{u}_{\text{des}}) \in (\mathcal{X} \ominus \mathcal{F}_{\infty}\times\ \mathcal{U} \ominus K\mathcal{F}_{\infty})\bigcap\mathcal{Z}_r \). In this case, it follows from \eqref{eq:ThetaDimond} that $M_1\theta^\diamond=\bar{x}_{\text{des}}$ and $M_2\theta^\diamond=\bar{u}_{\text{des}}$. Thus, from \eqref{eq:Proof4}, we also have $r=L\theta^\ast(t)$, $M_1\theta^\ast(t)=\bar{x}_{\text{des}}$ and $M_2\theta^\ast(t)=\bar{u}_{\text{des}}$ for $t\geq t^\circ$. 

Next, assume that \( (\bar{x}_{\text{des}}, \bar{u}_{\text{des}}) \not\in (\mathcal{X} \ominus \mathcal{F}_{\infty}\times\ \mathcal{U} \ominus K\mathcal{F}_{\infty})\bigcap\mathcal{Z}_r \). In this case, since $r=L\theta^\diamond$, it can be shown that\footnote{Given $z_1,z_2\in\mathbb{R}^n$ and $Q\succeq0$, we have: $\left\Vert z_2\right\Vert_Q^2-\left\Vert z_2\right\Vert_Q^2=\left\Vert z_2\right\Vert_Q^2-2\left\Vert z_2\right\Vert_Q^2+\left\Vert z_2\right\Vert_Q^2-2z_1^\top Qz_2+2z_1^\top Qz_2=\left\Vert z_1-z_2\right\Vert_Q^2+2(z_1^\top-z_2^\top)Qz_2\leq\overline{\lambda}(Q)\left\Vert z_1-z_2\right\Vert^2+2\left\Vert z_1-z_2\right\Vert\left\Vert Q\right\Vert\left\Vert z_2\right\Vert$, where $\overline{\lambda}(Q)$ is the largest eigenvalue of matrix $Q$.}:
\begin{align}
&\underline{\lambda}(L^\top Q_rL)\left\Vert\theta^\diamond-\theta^\ast(t)\right\Vert^2\leq\overline{\lambda}(M_1^\top Q_{sx}M_1)\left\Vert\theta^\diamond-\theta^\ast(t)\right\Vert^2\nonumber\\
&+\overline{\lambda}(M_2^\top Q_{su}M_2)\left\Vert\theta^\diamond-\theta^\ast(t)\right\Vert^2\nonumber\\
&+2\left\Vert M_1\right\Vert\left\Vert\theta^\diamond-\theta^\ast(t)\right\Vert\left\Vert Q_{sx}\right\Vert\left\Vert M_1\theta^\ast(t)-\bar{x}_{\text{des}}\right\Vert\nonumber\\
&+2\left\Vert M_2\right\Vert\left\Vert\theta^\diamond-\theta^\ast(t)\right\Vert\left\Vert Q_{su}\right\Vert\left\Vert M_2\theta^\ast(t)-\bar{u}_{\text{des}}\right\Vert,\label{eq:Proof5}
\end{align}
where $\underline{\lambda}(\cdot)$ and $\overline{\lambda}(\cdot)$ indicate the smallest and largest eigenvalues, respectively. Thus, if $Q_r$ is selected such that $\underline{\lambda}(L^\top Q_rL)>\overline{\lambda}(M_1^\top Q_{sx}M_1)+\overline{\lambda}(M_2^\top Q_{su}M_2)$, \eqref{eq:Proof5} implies that $\left\Vert\theta^\diamond-\theta^\ast(t)\right\Vert$ satisfies the bound \eqref{eq_bound} with $f(\theta)$ and $\alpha$ as: 
\begin{align}
f(\theta)=&2\left\Vert M_1\right\Vert\left\Vert Q_{sx}\right\Vert\left\Vert M_1\theta-\bar{x}_{\text{des}}\right\Vert\nonumber\\
&+2\left\Vert M_2\right\Vert\left\Vert Q_{su}\right\Vert\left\Vert M_2\theta-\bar{u}_{\text{des}}\right\Vert,
\end{align}
and 
\begin{align}\label{eq:alpha}
\alpha=\underline{\lambda}(L^\top Q_rL)-\overline{\lambda}(M_1^\top Q_{sx}M_1)-\overline{\lambda}(M_2^\top Q_{su}M_2).
\end{align}
\end{proof}


\begin{remark}\label{remark:UpperBound}
Since $Q_r$, $Q_{sx}$, and $Q_{su}$ are design parameters, the upper-bound in \eqref{eq_bound} can be made arbitrarily small. As a result, the output tracks $r$ and the steady-state configuration converges to the best admissible one. 
\end{remark}




\section{Simulation Studies}\label{sec:simulation}

The main goal of this section is to analyze the performance and effectiveness of the proposed RSSA-MPC scheme for position control of a Parrot Bebop 2 drone. 

Using a sampling period to 0.2 seconds, the position dynamics of the Parrot Bebop 2 drone can be expressed \cite{AmiriMECC2024,Momani2024} as a linear system in the form of \eqref{eq1}, where $x=[p_x ~ \dot{p}_x ~ p_y ~ \dot{p}_y ~ p_z ~ \dot{p}_z]^\top$ with $p_x,p_y,p_z\in\mathbb{R}$ being  X, Y, and Z positions in the global Cartesian coordinate,  $u=[u_x~u_y~u_z]^\top$  with $u_x,u_y,u_z\in\mathbb{R}$ being control inputs on X, Y, and Z directions, $C=I_6$, $D=\mathbf{0}$, and 
\begin{align*}
&A=\begin{bmatrix}
1 & 0.19895 & 0 & 0 & 0 & 0 \\
0 & 0.98952  & 0 & 0 & 0 & 0 \\
0 & 0 & 1.000 & 0.19963 & 0 & 0 \\
0 & 0 & 0 & 0.99627 & 0 & 0 \\
0 & 0 & 0 & 0 & 1.000 & 0.16816\\
0 & 0 & 0 & 0 & 0 & 0.69946
\end{bmatrix},\\
&B=\begin{bmatrix}
-0.10917348 & 0 & 0  \\
-1.08982035 & 0 & 0  \\
0 & -0.141040918& 0  \\
0 & -1.409531141 & 0  \\
0 & 0 & -0.030967224  \\
0 & 0 & -0.292295416 
\end{bmatrix}.
\end{align*}

The control inputs $u_x(t)$, $u_y(t)$, and $u_z(t)$ should satisfy the following constraints for all $t$: $|u_x(t)|\leq0.05$ , $|u_y(t)|\leq0.05$, and $|u_z(t)|\leq0.6$. Also, we assume that the position along Y direction is constrained as $|p_y(t)|\leq1.8$. We set $r = [1~0~2~0~1.5~0]^T$, the weighting matrices $Q_x=\text{diag}\{5,5,5,5,5,5\}$, $Q_u=\text{diag}\{30,20,1\}$, $Q_{sx}=\text{diag}\{0,0,0,0,0,0\}$, and $Q_{su}=\text{diag}\{1,1,1\}$, desired steady-state configuration  $\bar{x}_{\text{des}}=r$ and $\bar{u}_{\text{des}}=[0~0~0]^\top$, and the length of the prediction horizon is $N=10$.

We assume that $w(t)=[0~w_x(t)~0~w_y(t)~0~w_z(t)]^\top$, which implies that the disturbance signals $|w_x(t)|\leq\beta$, $|w_y(t)|\leq\beta$, and $|w_z(t)|\leq\beta$ affect the velocity along the X, Y and Z directions, respectively, with $\beta$ being the maximum magnitude of the disturbance in each direction. 


\begin{figure}[t]
    \centering
    \includegraphics{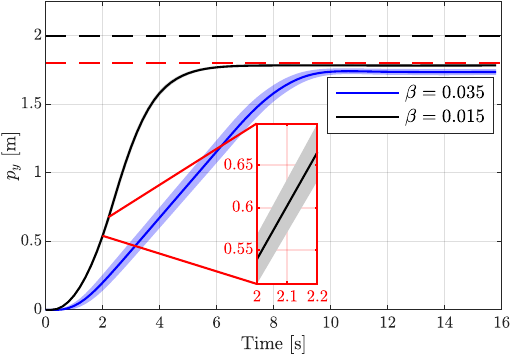}
    \caption{Time profile of $p_y$ at different disturbance magnitudes. The solid lines show the average response, the shaded regions represent the standard deviation, and the red and black dashed lines show the constraint and reference Y position, respectively.}
    \label{x3-time}
\end{figure}

\begin{figure}[t]
    \centering
    \includegraphics{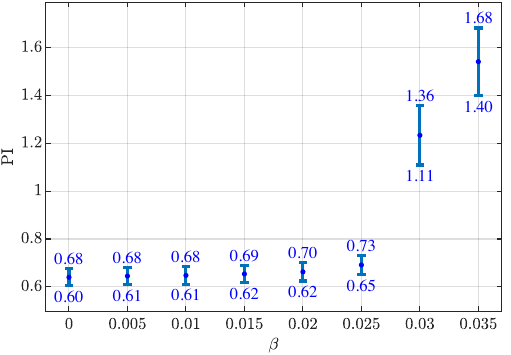}
    \caption{Performance Index (PI) as a function of disturbance magnitude $\beta$.}
    \label{tracking_error_dist}
\end{figure}


To provide a quantitative analysis, we consider 1,000 experiments with initial condition $x(0)=[-1~0~0~0~0.5~0]^\top$, wherein each experiment $w_x$, $w_y$, and $w_z$ are uniformly selected from the interval $[-\beta,\beta]$.

Fig. \ref{x3-time} illustrates the time-profile of $p_y$ for $\beta = 0.015$ and $\beta = 0.035$, where the solid lines represent the mean response, while the shaded areas denote the standard deviation. As shown in the figure, increasing the magnitude of the disturbance signal results in a slower response, reflecting the system's need to adjust to the disturbance. Furthermore, as expected, larger disturbances lead to larger uncertainty, which in turn increases the tracking error.

\begin{figure}[!t]
    \centering
    \includegraphics{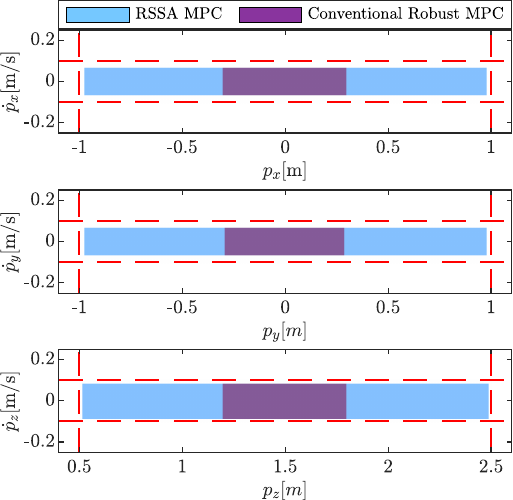}
    \caption{Region of Attraction for RSSA-MPC and the convectional robust MPC \cite{mayne2005robust} with $N=10$. The red dashed line represents the constraint.}
    \label{roa}
\end{figure}

This observation is illustrated in Fig. \ref{tracking_error_dist}, where Performance Index (PI) is $PI = \frac{1}{T} \sum_{t = 0}^{t = T} \| x(t) - \bar{x}_r \|^{2}$, with $T$ being the duration of the experiments. This figure shows the mean and standard deviation of PI for 1,000 experiments with initial conditions $x(0)=[\sigma_1~0~\sigma_2~0~\sigma_3~0]^\top$, where in each experiment $\sigma_1$, $\sigma_2$, and $\sigma_3$ are uniformly selected from the intervals [-0.5,0], [0,0.5], and [0.5,1], respectively. Fig. \ref{tracking_error_dist} reinforces the observation that larger disturbance magnitudes lead to increased uncertainty, resulting in a higher PI for the proposed RSSA-MPC.

We also analyze the Region of Attraction (ROA) of the proposed RSSA-MPC, defined as the set of initial conditions for which the MPC problem remains feasible. Fig. \ref{roa} describes the ROA for both the proposed RSSA-MPC and the conventional robust MPC \cite{mayne2005robust} with $N=10$, calculated under identical conditions, including a disturbance magnitude of $\beta = 0.02$, state constraints indicated by the red dashed lines, and the previously mentioned control input constraints. As shown in Fig. \ref{roa}, the proposed RSSA-MPC demonstrates a larger ROA compared to the conventional robust MPC outlined in \cite{mayne2005robust}.


\section{Experimental Results}\label{sec:Experimental}
This section aims to experimentally validate the proposed RSSA-MPC by evaluating its performance in controlling the position of the Parrot Bebop 2 drone.

Our experimental setup is shown in Fig. \ref{fig:Network}. We use the \texttt{OptiTrack} system with ten \texttt{Prime$^\text{x}$ 13} cameras operating at a frequency of 120 Hz. These cameras provide 3D accuracy within $\pm0.02$ millimeters. The computing unit consists of a 13th Gen $\text{Intel}^{\text{\textregistered}}$ $\text{Core}^{\text{\texttrademark}}$ i9-13900K processor with 64 GB of RAM, running the \texttt{Motive} software to analyze and interpret the camera data. We use the ``Parrot Drone Support from MATLAB" package \cite{MATLAB} to send control commands to the Parrot Bebop 2 via WiFi, utilizing the \texttt{move($\cdot$)} command. Communication between \texttt{Motive} and MATLAB is established via User Datagram Protocol (UDP) using the \texttt{NatNet} service.



Experimental results for $\beta = 0.015$ and $\beta = 0.035$ are shown in Fig.~\ref{real_drone_015} and Fig.~\ref{real_drone_035}, respectively, illustrating both the drone's position and control inputs. In these experiments, disturbances are introduced as variations in the control inputs. As demonstrated, the proposed RSSA-MPC effectively guides the Parrot Bebop 2 drone to the desired location while consistently satisfying state and input constraints. The results highlight the robustness of the RSSA-MPC against external disturbances.

\begin{figure}[!t]
    \centering
    \includegraphics[width=8.5cm]{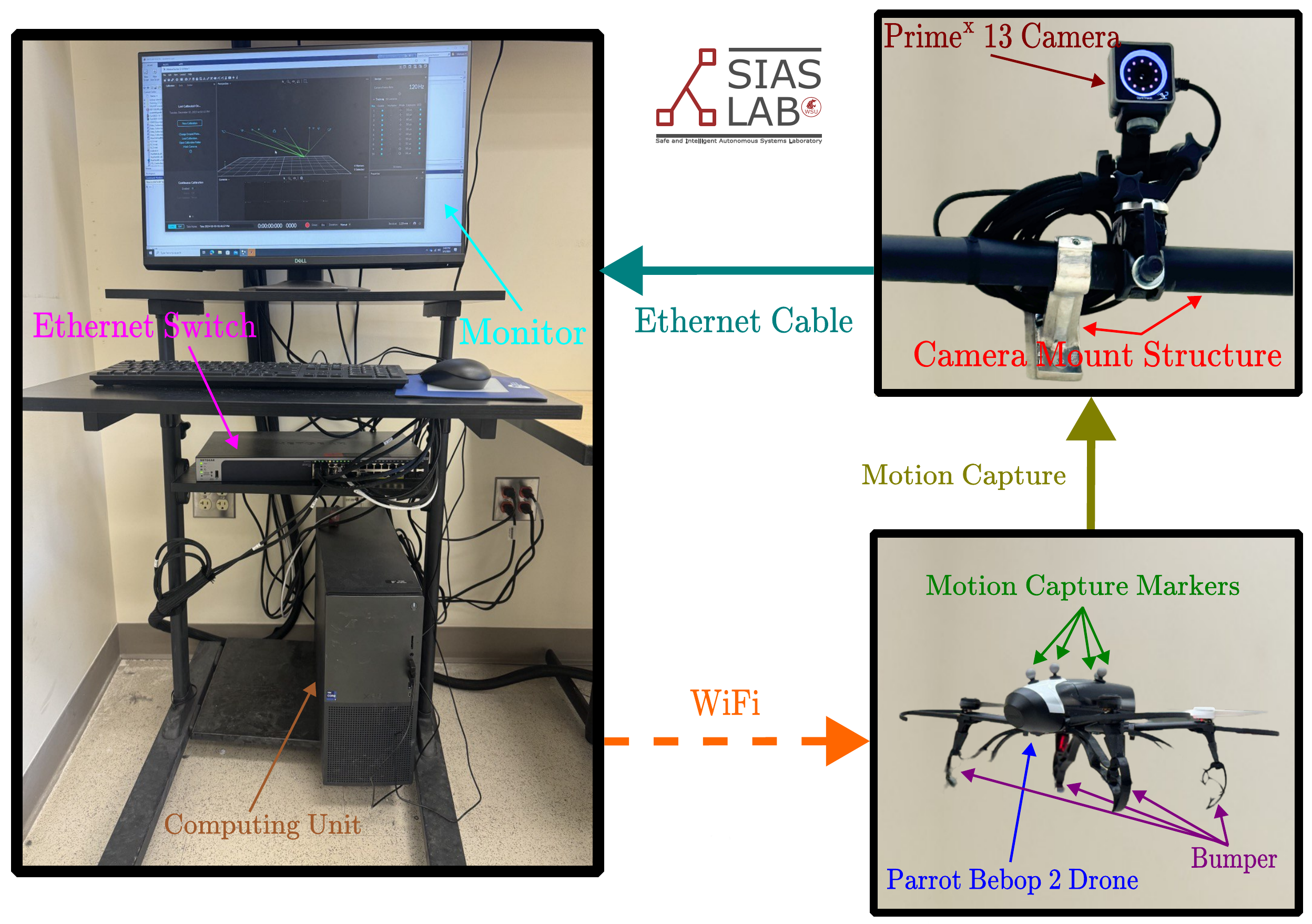}
    \caption{Overview of the experimental setup utilized to experimentally validate the proposed RSSA-MPC.}
    \label{fig:Network}
\end{figure}

\begin{figure*}[!t]
     \centering
     \includegraphics{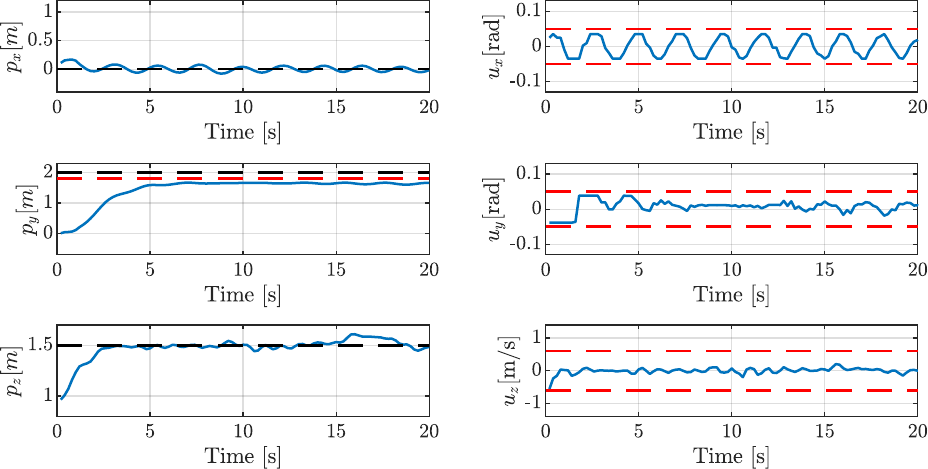}
     \caption{Time profiles of the drone's X, Y, and Z positions and control inputs for $\beta = 0.015$, with red dashed lines indicating constraints and black dashed lines representing reference positions.}
     \label{real_drone_015}
\end{figure*}

\begin{figure*}[!t]
     \centering
     \includegraphics{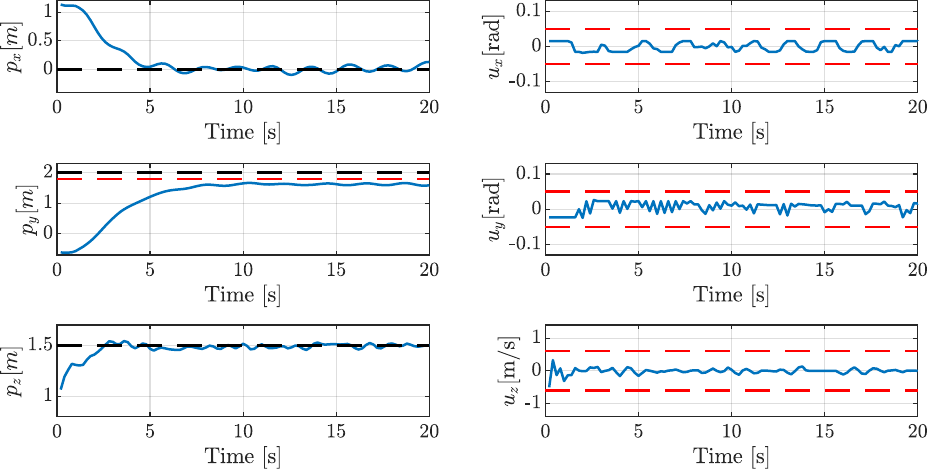}
     \caption{Time profiles of the drone's X, Y, and Z positions and control inputs for $\beta = 0.035$, with red dashed lines indicating constraints and black dashed lines representing reference positions.}
     \label{real_drone_035}
\end{figure*}

\section{Conclusion}\label{sec:conclusion}
In this paper, we introduced an enhancement to the steady-state-aware MPC which has been developed to mediate between system performance and constraint satisfaction in systems with limited computing capacity. More precisely, we adopted a tube-based analysis to introduce robustness to the steady-state-aware MPC, leading to the development of the RSSA-MPC scheme. It has been shown that the proposed methodology effectively decouples nominal trajectory optimization from robust control synthesis, maintaining the same computational complexity as conventional steady-state-aware MPC; thus, it remains appropriate for systems with limited computational resources. Analytical guarantees were provided, which demonstrate the theoretical validity of the proposed method. Simulations and experimental validation on a Parrot Bebop 2 drone demonstrated the robustness and effectiveness of the proposed RSSA-MPC, highlighting its applicability to real-world systems subject to disturbances.



\bibliographystyle{ieeetr}
\bibliography{references}

\end{document}